\documentclass[amstex,10pt]{article}
\usepackage{amssymb,latexsym}
\usepackage{amsmath,amsthm}
\usepackage{graphicx}
\usepackage{subfigure}
\usepackage{caption}
\usepackage{mathtools}
\usepackage{array,color,colortbl}

\newcounter{case}
\newcounter{subcase}[case]
\newtheorem{theorem}{Theorem}[section]

\newtheorem{conjecture}[theorem]{Conjecture}
\newtheorem{convention}[theorem]{Convention}

\theoremstyle{remark}
\newtheorem{notation}[theorem]{Notation}
\newtheorem{remark}[theorem]{Remark}

\DeclarePairedDelimiter\ceil{\lceil}{\rceil}
\DeclarePairedDelimiter\floor{\lfloor}{\rfloor}

\newtheorem{observation}[theorem]{Observation}

\renewcommand{\geq}{\geqslant}
\renewcommand{\leq}{\leqslant}
\renewcommand{\ge}{\geqslant}
\renewcommand{\le}{\leqslant}

\def\cref#1{Corollary~$\ref{#1}$}

\usepackage{array,color,colortbl}

\newcommand{\cc}{\mathcal{C}}

\newcommand{\ci}{\mathcal{I}}

\newcommand{\cm}{\mathcal M}

\title{On a conjecture of Stein}

\author{
Ron Aharoni\\
\small Department of Mathematics, Technion, Haifa 32000, Israel\\
\and Eli Berger \\ \small Department of Mathematics, Haifa University, Haifa 31999, Israel\\
\and Dani Kotlar\\
\small Department of Computer Science, Tel-Hai College, Upper Galilee, Israel\\
\and Ran Ziv\\
\small Department of Computer Science, Tel-Hai College, Upper Galilee, Israel\\
}

\date{}

\begin{document}

\maketitle

\maketitle

\begin{abstract}
Stein \cite{stein} proposed the following conjecture: if the edge set of $K_{n,n}$ is partitioned into $n$ sets, each of size $n$, then there is a partial rainbow matching of size $n-1$. He proved that there is a partial rainbow matching of size $n(1-\frac{D_n}{n!})$, where $D_n$ is the number of derangements of $[n]$. This means that there is a partial rainbow matching of size about $(1- \frac{1}{e})n$.  Using a topological version of Hall's theorem we improve this bound to $\frac{2}{3}n$.

 \end{abstract}

\section{Introduction}

A {\em Latin square} of order $n$ is an $n \times n$ array, in which each row and each column is  a permutation of $\{1,\ldots, n\}$.
A {\em partial transversal} in an $n \times n$ array of symbols is a set of  entries, each in a distinct row and distinct column, and having distinct symbols. If the partial transversal is
of size $n$, then it is called a {\em full transversal}, or simply a {\em transversal}.
In $1967$ Ryser \cite{ryser} published a conjecture that has  since gained some renown:

\begin{conjecture}\label{ryser}
An odd Latin square has a transversal.
\end{conjecture}

For even $n$ the Latin square defined by $L(i,j)=i+j \pmod n$ does not have a transversal. But a natural conjecture is:

\begin{conjecture}\label{sb}
An  $n\times n$ Latin square has a partial transversal of size $n-1$.
\end{conjecture}

This was conjectured independently by Stein \cite{stein} and by Brualdi \cite{brualdiryser}. An even more general conjecture was made by Stein:

\begin{conjecture}\label{steingeneral}
Let $A$ be an $n \times n$ array of symbols, each symbol appearing in precisely $n$ entries. Then $A$ has a partial transversal of size $n-1$.
\end{conjecture}

In this conjecture there is no distinction between the cases $n$ odd and $n$ even. For every   $n>1$  there exist arrays as in the conjecture, not having a full transversal (take $A(i,j)=i$ for $j <n$, $A(i,n)=i+1 \pmod n$).
Partial results on all three conjectures, as well as on related conjectures, abound. Koksma \cite{koksma} proved   that in an $n \times n$ Latin square there is a partial transversal of size at least $\frac{2}{3}n$. Woolbright \cite{woolbright}  improved this bound to  $n-\sqrt{n}$, and Shor and Hatami (\cite{shor} with a correction in \cite{shorhatami}) proved the best bound known so far,  namely, $n - 11\log_2^2 n$.

As to Conjecture \ref{steingeneral}, Stein himself proved that on average a permutation submatrix contains at least $n(1-\frac{D_n}{n!})$ distinct symbols, where $D_n$ is the number of derangements of $[n]$. Since $D_n$ is asymptotically $\frac{n!}{e}$,  this yields a bound of about $(1-\frac{1}{e})n$. In this paper we improve the bound to $\frac{2}{3}n$.

\section{The special ISR properties of line graphs}

Let $V_1, V_2, \ldots ,V_m$ be a partition of the vertex set $V$ of a graph $G$. A {\em  partial ISR} (Independent Set of Representatives) for this system is an independent set in $G$, meeting each  $V_i$ in at most one vertex. The word ``partial'' is omitted if it is of size $m$. So, an ISR is a full choice function of the $V_i$'s, whose image is independent in $G$. Other terms used in the literature are ``independent transversal'' and ``rainbow independent set''. When the graph $G$ is a line graph of a graph $H$ then an ISR is called a ``rainbow matching'' (since it is a matching in $H$).\\

Broadly speaking, the sparser a graph is the more likely it is to have  ISRs. This explains the logic behind the following theorem of Haxell \cite{pennyfirst}:

\begin{theorem}\label{basic}
If $|V_i| \ge 2\Delta(G)$ for all $i=1,\ldots,m$, then there exists an ISR.
\end{theorem}

It appears that line graphs behave particularly well with respect to ISRs. For example, the following can be proved from the results of \cite{ah}:

\begin{theorem}\label{deltah}
If $G=L(H)$, for some graph $H$, and $|V_i| \ge 2\Delta(H)$ for all $i=1,\ldots,m$, then there exists a rainbow matching.
\end{theorem}

In \cite{aab} this was improved to:

\begin{theorem}\label{aabtheorem}
If $V_1, V_2, \ldots ,V_m$ is a partition of the vertex set of the line graph $G=L(H)$ of any  graph $H$ and  $|V_i| \ge \Delta(G)+2$ for all $i=1,\ldots,m$, then the system has an ISR.
\end{theorem}

Clearly, $\Delta(G)+2 \le 2\Delta(H)$. In \cite{agn} this was generalized to hypergraphs. For a hypergraph $H$ denote by $\tilde{\Delta}(H)$ the maximum, over all edges $e \in H$, of $\sum_{v \in e}deg_H(v)$.

\begin{theorem}\label{agntheorem}
If $V_1, V_2, \ldots ,V_m$ is a partition of the vertex set of the line graph $G=L(H)$ of a hypergraph $H$ and  $|V_i| \ge \tilde{\Delta}(H)$ then the system has an ISR.
\end{theorem}

To see why Theorem  \ref{agntheorem} generalizes Theorem \ref{aabtheorem} note that if $H$ is $r$-uniform and linear (no two edges meet at more than one vertex), then
$\Delta(G) = \max_{e \in H} \sum_{v \in e}deg_H(v)-r$, meaning that $\tilde{\Delta}(H)=\Delta(G)+r$.

 Stein's conjecture can be viewed as expressing the special ISR properties of line graphs. In ISR terminology (rainbow matchings, in this case), it is:

\begin{conjecture}\label{steinknn}
Any partition $V_1, \ldots ,V_n$  of $E(K_{n,n})$ into sets of size $n$ has a partial rainbow matching of size $n-1$.
\end{conjecture}

If true, this would clearly imply a conjecture that has the advantage of being negotiable in small cases:

\begin{conjecture}\label{equirep}
Suppose that  the edges of $K_{n,n}$ are partitioned into
 sets $E_1, E_2, \ldots ,E_m$ ($m\le n$). Then there exists a perfect matching $F$
 in $K_{n,n}$ satisfying $|F \cap E_i| \ge \floor* {\frac{|E_i|}{n}
 }-1$, with strict inequality holding for all but one value
 of $i$.
\end{conjecture}

In \cite{abcklz} this was proved for $m=3$, using Sperner's lemma. The following  conjecture suggests that line graphs of bipartite graphs behave not only twice, but possibly $4$ times, better than ordinary graphs with respect to ISRs:

\begin{conjecture}\label{deltaplus1}
If $V_1, V_2, \ldots ,V_m$ is a partition of  the edge set of the line graph $L(H)$ of a simple bipartite graph $H$ and $|V_i| > \Delta(H)+1$ then the system has an ISR.
\end{conjecture}

The inequality $|V_i| \ge \Delta(H)+1$ does not suffice, as is shown by a well known example \cite{jin,yuster}: $V_j=\{a_{i,j} \mid i\le 3\}$, ~~$j=1,2,3,4$, and $F$ consists of three $C_4$s: $a_{11}a_{12}a_{21}a_{22}$,~~$a_{13}a_{14}a_{23}a_{24}$ and $a_{31}a_{33}a_{32}a_{34}$.
 The simplicity of $H$ is necessary since in non simple graphs Theorem \ref{deltah} is sharp, as shown in the following example:
  $V_i, ~i=1,\ldots,k$ consists of pairs of edges, $e_i=(a_i,b_i),~~f_i=(c_i,d_i)$, each repeating $k$ times, and $V_0=\bigcup_{i \le k}\{(a_i,d_i),(b_i,c_i)\}$. Here $|V_i|=2k, \Delta(H)=k+1$ and there is no rainbow matching.

\section{A topological tool}

For a graph $G$ denote by $\ci(G)$ the complex (closed down hypergraph) of  independent sets in $G$.
If $G=L(H)$,  then $\ci(G)$ is the complex of matchings in $H$. We shall denote this complex by $\cm(H)$.

There is a topological version of Hall's theorem, that allows to extend Hall's theorem to the setting of ISRs. Here are some relevant preliminary notions.
 A simplicial complex  $\cc$ is called (homologically)
\emph{$k$-connected}  if for every $-1 \le j \le k$, the $j$-th reduced simplicial homology group of $\cc$ with rational coefficients $\tilde{H}_j(\cc)$ vanishes. The (homological) \emph{connectivity} $\eta_H(\cc)$ is the largest $k$ for which $\cc$ is $k$-connected, plus $2$.

\begin{remark}
\noindent
\begin{itemize}
\item[(a)] This is a shifted (by $2$) version of the usual definition of connectivity. This shift simplifies the statements below, as well as the statements of basic properties of the connectivity parameter.
\item[(b)] If $\tilde{H}_j(\cc)=0$ for all $j$ then we define  $\eta_H(\cc)$ as $\infty$.
\item[(c)] There exists also a homotopical notion of connectivity, $\eta_h(\cc)$: it is the minimal dimension of a ``hole'' in the complex.  The first topological version of Hall's theorem \cite{ah} used that notion.
The relationship between the two parameters is that $\eta_H\ge \eta_h$ for all complexes, and if $\eta_h(\cc)\ge 3$, meaning that the complex is simply connected, then $\eta_H(\cc)= \eta_h(\cc)$.  All facts mentioned in this article (in particular, the main tool we use, the Meshulam game) apply also to $\eta_h$.
\end{itemize}
\end{remark}

\begin{notation}
  Given a graph $G$, a partition $(V_i)_{i=1}^n$ of $V(G)$ and a subset $I$ of $[n]=\{1,\ldots,n\}$, we write $V_I$ for $\bigcup_{i\in
    I}V_i$. We denote by $\ci(G) \upharpoonright A$ the complex of independent sets in the graph induced by $G$ on $A$.
\end{notation}

The topological version of Hall's theorem is:

\begin{theorem}\label{ah}
  If $\eta_H(\ci(G) \upharpoonright V_I) \ge |I|$ for every $I \subseteq
  [n]$ then there exists an ISR.
\end{theorem}

Variants of this theorem appeared implicitly in \cite{ah} and \cite{me1}, and the theorem is stated and proved explicitly as Proposition~1.6 in \cite{me2}.

In order to apply Theorem~\ref{ah}, one needs good lower bounds on $\eta_H(\ci(G))$. One of the known bounds is  in terms of a graph parameter called the ``independence domination number'' and denoted by $i\gamma$. For a graph $G$, $i\gamma(G)$ is the maximum, over all independent sets $I \in \ci(G)$, of the minimal number of vertices needed to dominate $I$. In \cite{ah} the following was proved for the homotopic connectivity, and this was later extended in \cite{me1} to homological connectivity:

\begin{theorem}\label{igamma}
$\eta_H(\ci(G)) \ge i\gamma(G)$.
\end{theorem}

An easy corollary is:

\begin{equation}\label{nuover2}
\eta_H(\ci(G)) \ge \nu(G)/2\; ,
\end{equation}
where $\nu(G)$ is the maximal size of a matching in $G$.

A general lower bound on $\eta_H(\ci(G))$ is essentially due to Meshulam \cite{me2} and is conveniently expressed in terms of a
game between two players, CON and NON, on the graph $G$. CON wants
to show high connectivity, NON wants to thwart his attempt. At each
step, CON chooses an edge $e$ from the graph remaining at this
stage, where in the first step the graph is $G$. NON can then either

\begin{enumerate}
\item
delete $e$ from the graph (we call such a step a ``deletion''),

or

\item remove the two endpoints of $e$, together
  with all neighbors of these vertices and the edges incident to them, from the graph (we call such a step
  an ``explosion'', and denote by $G*e$ the resulting graph).

\end{enumerate}

\noindent The result of the game (payoff to CON) is defined as follows: if at some point there
remains an isolated vertex, the result is $\infty$. Otherwise, at some
point all vertices have disappeared, in which case the result of the
game is the number of explosion steps. We define $\Psi(G)$ as the value of the game, i.e., the result obtained by optimal play on the graph $G$.

\begin{convention} \label{conventionmin} Henceforth we shall assume that NON always chooses the best strategy for him, namely he removes $e$ if   $\min(\Psi(G-e), \Psi(G*e)+1)=\Psi(G-e)$, and explodes it if $\min(\Psi(G-e), \Psi(G*e)+1)=\Psi(G*e)+1$. \end{convention}

\begin{theorem}\label{etaPsi}
$\eta_H(\ci(G))\ge \Psi(G)$.
\end{theorem}

\begin{remark}
 The idea underlying this lower bound originated in \cite{me2}. The game theoretic formulation first appeared in \cite{abz}. This formulation of $\Psi$ is equivalent to a recursive definition of $\Psi(G)$ as the maximum over all edges of $G$, of the minimum between $\Psi$ of the graph obtained by deleting the edge, and $\Psi$ of the graph obtained by exploding it plus $1$. For an explicit proof of Theorem~\ref{etaPsi} using the recursive definition of $\Psi$, see Theorem~1 in \cite{adba}.

\end{remark}

A standard argument of adding dummy vertices yields the  deficiency version of Theorem \ref{ah}:
\begin{theorem}\label{deficiency}
If $\eta_H(\ci(G[\bigcup_{i \in I}V_i])) \ge |I|-d$ for every $I \subseteq [m]$ then the system has a partial  ISR of size $m-d$.
\end{theorem}

Our main result is:

\begin{theorem}\label{stein23}
A partition $V_1, \ldots ,V_n$  of $E(K_{n,n})$ into sets of size $n$ has a partial rainbow matching of size at least $\frac{2}{3}n- \frac{1}{2}$.

\end{theorem}

By Theorem \ref{deficiency} the theorem will follow from:

\begin{theorem}\label{eta}
If $F \subseteq E(K_{n,n})$ then

\begin{equation}\label{mainineq} \eta_H(\cm(F)) \ge \frac{|F|}{n}-\frac{n}{3} - \frac{1}{2}.\end{equation}
\end{theorem}

This is a generalization  of a theorem proved in \cite{bjornerlovaszziv}:

\begin{theorem}\label{blz}
$\eta_H(\cm(K_{n,n})) \ge \floor{ \frac{2n}{3} }$.
\end{theorem}

In \cite{sw} it was shown that in fact equality holds in Theorem \ref{blz}.

\section{Proof of Theorem \ref{eta}}

For a graph $G = (V,F)$ write $k(G) = \frac{|V|}{2}-\frac{2|F|}{|V|}$ if $V \neq \emptyset$, and $k(G)=0$ if $V = \emptyset$.
In order to prove Theorem \ref{eta} it suffices to show that for $G$ bipartite with sides of equal size,

\begin{equation}\label{maininequality} \eta_H(\ci(L(G))) \geq \frac{2|F|}{|V|}-\frac{|V|}{6}-\frac{1}{2} = \frac{|V|}{3}-  k(G)-\frac{1}{2}.\end{equation}

For this purpose we play the Meshulam game on $L(G)$. Note that playing the game on $L(G)$
means offering pairs $(xy, xz)$ of adjacent edges in $G$. Exploding such a pair is tantamount to removing $x,y,z$ from $G$ and
deleting such a pair (edge in $L(G)$) means separating the edges $xy$ and $xz$, namely making them non-adjacent in $L(G)$. Below we shall use the fact that for a given edge $e$,  if CON offers to NON the pairs $(e,f)$
for all edges $f$ adjacent to $e$, if NON does not explode any of these pairs then the edge $e$ becomes isolated as a vertex of $L(G)$, giving the game the value $\infty$.

We break the game into sequences of consecutive steps,
each ending with the emergence of an isolated edge, meaning that the value of the game is $\infty$ (and then the game ends), or with an explosion.
The graph resulting from the $i$th sequence is denoted by $G_{i+1}$.
The construction will be so defined, that $G_{i+1}$ is obtained from $G_i$ by the removal of two or three vertices (this will be shown below). In particular, $G_{i+1}$ will be an induced subgraph of $G_i$.
Let  $G_0=G, G_1, \ldots, G_t$ be the sequence of graphs obtained. The process ends at a graph $G_t$ that has no edges. The fact that $G_{i+1}$ is obtained from $G_i$ by the removal of vertices means that
all graphs  $G_i$ are induced subgraphs of $G$. For $i \le t$ we shall denote $V(G_i)$ by $V_i$,  and $|V_i|$ by $n_i$.

Let $(P_0,Q_0)$ be the given bipartition
of the graph, and let $P_i=P_0\cap V_i,~~Q_i=Q_0\cap V_i$. Let $|P_i|=p_i, |Q_i|=q_i$.
The process will be so defined that
$|p_i-q_i| \le 1$.

We shall also maintain a naming $X_i,~Y_i$ to the two sides, by the following rules.
Let $X_0=P_0, Y_0=Q_0$.
  If $p_i\neq q_i$ then $X_i$ is
the larger of $P_i,~Q_i$ and $Y_i$ is the smaller of the two.
If $p_i=q_i$ and $i>0$, then  $X_i=Q_i$. The process will be so defined, that when removing three vertices, two vertices will be removed from $X_i$ and one vertex will be removed from $Y_i$. This implies:

\begin{equation}\label{balanced}
0 \le |X_i|-|Y_i|\le 1\; .
\end{equation}

By \eqref{balanced} (used for both $i$ and $t$):

\begin{equation*}\label{notfar}
-1 \le |X_i  \setminus V_t|- |Y_i  \setminus V_t|\le 1.
\end{equation*}

Since $|X_i  \setminus V_t|+ |Y_i  \setminus V_t| =n_i-n_t$, this entails:

\begin{equation}\label{smallworld}
\max(|X_i  \setminus V_t|, |Y_i  \setminus V_t|) = \ceil*{\frac{n_i-n_t}{2}}
\end{equation}

The inductive construction goes as follows. The graph  $G_i$ having been defined, we  choose a vertex $v$ in $G_i$, considering two possibilities:\\

POS1.~ There exists an isolated vertex in $X_i$. Let $v$ be a vertex of minimal positive degree in $G_i$.

POS2.~ There is no isolated vertex in $X_i$. In this case let $v$ be a vertex of minimal positive degree in $X_i$. \\

In both cases we denote the degree of $v$ in $G_i$ by $\delta_i$.

 Let $e=xy$ be an edge in $G_i$ containing $v$, where $x \in X_i$ and $y \in Y_i$. We denote by $d(y)$ the degree of $y$ in $G_i$. We distinguish between two cases: \textbf{Case I:} $d(y)>1$ and \textbf{Case II:} $d(y)=1$. 
 % - so $x=v$ if $v \in X_i$ and $y=v$ if $v \in Y_i$. Call $w$ the vertex of $e \setminus \{v\}$.
We first consider Case I. In this case CON  offers NON the pairs $(e,f)$
(each such pair corresponds to an edge in $L(G)$), for all edges $f\neq e$ containing $y$, in  an arbitrary order.
Note that $e \cup f$ contains two vertices in $X_i$ and one vertex in $Y_i$.
 If NON explodes one of these pairs, the $i$th sequence of steps is thereby completed, and the graph $G_{i+1}$ is obtained.

If not, then NON has separated all pairs $(e,f)$, for all edges $f \neq e$ containing $y$, namely he removed the $L(G)$-edges connecting them. In this case (which we call ``lagging'')
 CON proceeds by offering  NON all pairs $(e,f)$ for all edges $f \neq e$ containing $x$.
 To avoid isolating $e$ (regarded as a vertex in $L(G)$), NON must explode a pair $(e,f)$, where, say, $f=xz$.
Since in this case all edges containing $y$ have been previously separated in $L(G)$ from $e$, they are not removed by such an explosion. The effect of the explosion is then that of removing $x$ and $z$ from $G_i$, while keeping the vertex $y$. Thus, so far only two vertices have been removed, one in $X_i$ and one in $Y_i$.

Still considering the lagging case, if POS1 above applies,  remove also a vertex in  $X_i$  that is isolated in $G_i$ (such exists, in this case. Note that this removal does not change $\eta(\ci(L(G_i)))$ ). Also note that in this case we  call ``$X_{i+1}$'' the opposite side to that of $X_i$, since it is this side that is bigger. If POS2 applies, then we do not remove a third vertex.

Now consider Case II. In this case CON offers NON pairs $(e,f)$ for all edges $f\neq e$ containing $x$. Note that such pairs exist, since the degree of $x$ must be greater than 1, otherwise the edge $e$ is isolated. Note also that NON must explode one of these pairs, otherwise $e$ will become isolated. In such an explosion two vertices from $Y_i$ and only one vertex from $X_i$ are removed, which may result in violation of (\ref{balanced}) for $X_{i+1}$ and $Y_{i+1}$. To avoid this, the isolated vertex $y$ is ``returned'' to the game, an action that does not affect the final score. If POS1 applies, remove an isolated vertex in $X_i$, so that a total of three vertices are removed. If POS2 applies, no further action is needed, so that only two vertices are removed in this step. 

Since in each step at most three vertices are deleted, we have

\begin{equation}\label{third} 
t \ge \frac{n_0-n_t}{3}\;
\end{equation}

Let $k_i = k(G_i)$. Since the value of the game played is $t$, Theorem \ref{eta} will follow from \eqref{maininequality} if we prove that
$t \geq  \frac{|V(G_0)|}{3}- k_0 -\frac{1}{2}$.
Thus we may assume for contradiction that

\begin{equation}
\label{forcontradiction}
t < \frac{n_0}{3}-  k_0 - \frac{1}{2} \;
\end{equation}

Combining \eqref{third} and \eqref{forcontradiction} yields

\begin{equation}\label{n0-nt:ineq}
\frac{n_0}{3}- \frac{n_t}{3} < \frac{n_0}{3}-  k_0 - \frac{1}{2}.
\end{equation}

Note that $k_0=k(G)=\frac{|V|}{2}-\frac{2|F|}{|V|}\ge \frac{2n}{2}-\frac{2n^2}{2n} = 0$. This, together with \eqref{n0-nt:ineq}, implies

\begin{equation}\label{ntlarge}
n_t \ge 2\; .
\end{equation}

Since $G_i$ is an induced subgraph of $G$ for all $i\le t$, and since $G_t$ is void of edges, we have

\begin{observation}\label{obs1}
$V_t$ is independent in every $G_i$.
\end{observation}

This and the definition of $\delta_i$ imply:
\begin{observation}\label{noisolatedinx}
If there exists a  vertex
 in  $X_i \cap V_t$ that is non-isolated in $G_i$ then $\delta_i \le |Y_i \setminus V_t|$.
\end{observation}

\begin{observation}\label{deltasmall}
$\delta_i \leq \ceil*{\frac{n_i-n_t}{2}}$.
\end{observation}

  \begin{proof} By Observation \ref{noisolatedinx} and \eqref{smallworld} the only case in which the observation may fail is that in which all vertices in $X_i \cap V_t$ are isolated. But  then
all vertices in $Y_i$ have degree
 at most $|X_i \setminus V_t|$, and the observation  follows again by \eqref{smallworld}.\end{proof}

Note that if three vertices are removed, two of the three vertices removed are in $X_i$ (possibly one of those being $v$, the vertex with degree $\delta_i$), hence the sum of the degrees of the three vertices is at most $|X_i|+|Y_i|+\delta_i=n_i+\delta_i$. In the case where two vertices $v,y$ are removed (this happens only in POS2),
$deg_{G_i}(y) \le |X_i|$, and
 the sum of the degrees of the two vertices is at most $|X_i|+\delta_i$. Hence, denoting $|E(G_i)|$ by $e_i$, we have:

\begin{equation}\label{howmanyremoved}
e_i-e_{i+1} \le
\begin{cases}
n_i+\delta_i-2 & \mbox{if three vertices are removed} \\
|X_i|+\delta_i -1& \mbox{if two vertices are removed}
\end{cases}
\end{equation}

The ``$-2$'' term in Case 1 compensates for the two edges $e$ and $f$ being counted twice. The ``$-1$'' term in Case 2 compensates for the edge $e$ being counted twice. 

 %To complete them to three vertices, choose a third vertex from $X_i$ with degree either 0 or $\delta_i$, and remove it.

%We say that a side of the bipartition is {\em allowed} if this side has no isolated vertices and it is larger than the other side.
%If the side have equal size or if the larger side has isolated vertices then both sides are allowed.
%We write $\delta'(G)$ for the minimal degree of a non-isolated vertex in an allowed side.

Now, \eqref{n0-nt:ineq} gives

\begin{equation}\label{bound:k_t}
    k_t = \frac{n_t}{2} > \frac{3}{2} k_0 + \frac{3}{4} \; .
\end{equation}

In particular, $k_t > k_0$.
Let $j$ be minimal
such that $k_j \leq k_0$ and $k_{j+1} > k_0$. The first of these facts means that
$\frac{n_j}{2}-\frac{2e_j}{n_j} \leq k_0$. Shuffling terms we get:

\begin{equation}\label{almostthere}
n_j(n_j- 2k_0) \leq 4e_j
\end{equation}

By Observation \ref{deltasmall} and (\ref{bound:k_t}),

\begin{equation}\label{bound:delta}
    \delta_j \leq n_j/2 - k_t + 1/2 < \frac{n_j}{2} - \frac{3}{2}k_0 - \frac{1}{4}
\end{equation}

We first consider the case in which three vertices are removed
in the passage from $G_j$ to $G_{j+1}$. Combining (\ref{bound:delta}) with \eqref{howmanyremoved} gives

$$ e_{j+1} \geq e_j -  n_j - \delta_j +2  \geq  e_j -  \frac{3}{2}n_j + \frac{3}{2}k_0 + 2 + \frac{1}{4}\; $$

 which upon multiplying by $4$ becomes:

$$ 4e_{j+1} \geq  4 e_j -  6n_j + 6k_0 + 9  $$

This, along with \eqref{almostthere}, yields

$$ 4e_{j+1} \geq n_j(n_j- 2k_0) -  6n_j + 6k_0 + 9 = (n_j-3)((n_j-3)- 2k_0)  = n_{j+1}(n_{j+1} - 2k_0).$$

This in turn can be written as:
\begin{equation}\label{final_eq}
    \frac{n_{j+1}}{2} - \frac{2e_{j+1}}{n_{j+1}} \leq k_0
\end{equation}

meaning that $k_{j+1} \leq k_0$, contrary to the assumption.

\medskip

Next consider  the case in which only two vertices are removed in the passage from $G_j$ to $G_{j+1}$.
Remember that this can happen only in POS2.  By
\eqref{ntlarge} and by \eqref{balanced} applied to $i=t$ we know that $|X_j \cap V_t| \ge 1$. By Observation \ref{noisolatedinx} this implies:

\begin{equation}\label{smallydeg}
\delta_j \le |Y_j  \setminus V_t|
\end{equation}

\begin{observation}\label{ejejplus1}
   $ e_j-e_{j+1}\le  n_j-\frac{n_t}{2}-\frac{1}{2}$.
\end{observation}

\begin{proof}

 {\bf Subcase I:} $n_j\equiv n_t \pmod{2}$.

 Since $n_j\equiv n_t \pmod{2}$ we have $\delta_j\le \frac{n_j-n_t}{2}$, by Observation~\ref{deltasmall}. Thus, by (\ref{howmanyremoved}) we have $e_j-e_{j+1}\le |X_j|+\delta_j-1\le \frac{n_j+1}{2}+\frac{n_j-n_t}{2}-1= n_j-\frac{n_t}{2}-\frac{1}{2}$, as desired.

{\bf Subcase II:}  $n_j\not\equiv n_t \pmod{2}$.

Suppose first that $n_j$ is odd. Then, $n_t$ is even and we have $|Y_j|<|X_j|$ and $|Y_t|=|X_t|$. Thus $|Y_j\setminus V_t|<|X_j\setminus V_t|$. By \eqref{smallydeg} we have $\delta_j\le |Y_j\setminus V_t|<\frac{n_j-n_t}{2}$. As before, we obtain $e_j-e_{j+1}<  n_j-\frac{n_t}{2}-\frac{1}{2}$.

 Suppose next that $n_j$ is even. Then, $n_t$ is odd and we have $|Y_j|=|X_j|$ while $|Y_t|<|X_t|$. Thus $|Y_j\setminus V_t|>|X_j\setminus V_t|$. By \eqref{smallydeg} we have $\delta_j\le |Y_j\setminus V_t|=\frac{n_j-n_t+1}{2}$. Since $|Y_j|=|X_j|=n_j/2$ we have $e_j-e_{j+1}\le |X_j|+\delta_j -1\le \frac{n_j}{2}+\frac{n_j-n_t+1}{2}-1=  n_j-\frac{n_t}{2}-\frac{1}{2}$.
\end{proof}

Combined this with (\ref{bound:k_t}) and the fact that $k_0 \ge 0$ this yields:

$$ e_{j+1} \geq e_j -  n_j +\frac{n_t}{2} + \frac{1}{2} > e_j -  n_j +\frac{3}{2}k_0 + \frac{5}{4} \;. $$

Multiplying this by 4 yields

$$ 4e_{j+1} \geq  4e_j - 4n_j + 6k_0 + 5\; . $$

Combining this with \eqref{almostthere}  we obtain,
\begin{equation*}
\begin{split}
  4e_{j+1} & \ge n_j(n_j- 2k_0)-4n_j + 6k_0+5 \\
    & = n_j^2-2 k_0 n_j - 4 n_j +6k_0+5 \\
    & > n_j^2-2 k_0 n_j - 4 n_j +4k_0 + 4 \\
    & = (n_j-2)((n_j-2)-2k_0) \\
    & = n_{j+1}(n_{j+1} - 2k_0)\; ,
\end{split}
\end{equation*}

which leads to (\ref{final_eq}) and thus yielding the desired contradiction to \eqref{forcontradiction}.

%{\bf Acknowledgement} We are indebted to Ron Holzman for a helpful discussion.

\end{document}